\documentclass[letterpaper, reqno,11pt]{article}
\usepackage[margin=1.0in]{geometry}
\usepackage{color,latexsym,amsmath,amssymb, amsthm, graphicx, enumitem}
\usepackage[percent]{overpic}
\usepackage{url}

\newcommand{\RR}{\mathbb{R}}

\newcommand{\ZZ}{\mathbb{Z}}

\newcommand{\eps}{\varepsilon}
\newcommand{\tubes}{\mathbb{T}}

\newtheorem{thm}{Theorem}
\newtheorem{lem}[thm]{Lemma}
\newtheorem{defn}[thm]{Definition}
\newtheorem{prop}[thm]{Proposition}
\newtheorem{cor}[thm]{Corollary}

\newtheorem*{twoEndsAss}{Assertion $\mathbf{TE}(d,a,b)$}
\newtheorem*{robustTransAss}{Assertion $\mathbf{RT}(d,a,b)$}

\newtheorem{rem}[thm]{Remark}

\begin{document}
\pagenumbering{arabic}
\title{New bound on the Hausdorff dimension of Kakeya sets in $\RR^4$}
\author{Mukul Rai Choudhuri}

\maketitle

\begin{abstract}
    We obtain a slight improvement to the trilinear bound of Guth-Zahl by application of Wang-Zahl's resolution of 3-D Kakeya to tubes inside a thickened neighborhood of a variety. This is combined with the planebrush argument of Katz and Zahl to obtain the new bound of 3.073.
\end{abstract}

\section{Introduction}

A Besicovitch set (also called a Kakeya set) in $\RR^n$ is a compact subset of $\RR^n$ containing a unit line segment in every direction. The Kakeya conjecture asserts that every Besicovitch set in $\RR^n$ must have Hausdorff dimension $n$. For $n=2$, the conjecture was proved by Davies in 1971 \cite{Davies}. The 3-dimensional case was proved by Wang and Zahl \cite{WZ1,WZ2,WZ3} in a major recent breakthrough. The conjecture remains open for $n\geq 4$, though partial results exist. See \cite{Z3} for a recent survey of progress on the Kakeya conjecture by Zahl.

In this paper, we make some partial progress towards the Kakeya conjecture in $\RR^4$. Specifically, we will prove the following result.

\begin{thm}\label{mainresult}
Every Besicovitch set in $\RR^4$ has Hausdorff dimension at least $\frac{13}{4}-\frac{12951-54\sqrt{10593}}{41972}=3.073854...$.
\end{thm}

Let us briefly discuss the recent history of progress on the Kakeya conjecture in $\RR^4$. In 2018, Guth and Zahl \cite{GZ} (in conjunction with the results of Zahl \cite{Z}, and independently Katz-Rogers \cite{KR}) proved that Kakeya sets in $\RR^4$ have Hausdorff dimension at least 3.025. The primary ingredients used are the polynomial Wolff axioms and grains decomposition, which were used to obtain a new trilinear Kakeya estimate.

Following this, Katz and Zahl \cite{KZ} improved the Hausdorff dimension bound to 3.059 using the planebrush argument. This was a higher dimensional analogue of Wolff's hairbrush argument \cite{W}. This in fact yields a Hausdorff dimension bound of 3, which was the longstanding record before the bound of Guth and Zahl \cite{GZ}. Katz and Zahl \cite{KZ} also used the trilinear bound of Guth-Zahl in their induction. A finite field version of the planebrush argument can be found in the expository paper \cite{LRZ}.

For our result, we will be obtaining a slight improvement to the trilinear bound of Guth-Zahl \cite{GZ} by application of Wang-Zahl's resolution of 3-D Kakeya \cite{WZ3}. This is then plugged in to the induction of Katz and Zahl \cite{KZ} to obtain the new bound of 3.073.

We will be working in the ambient space $\RR^4$ throughout this paper.

\subsection{Acknowledgement}

The author would like to thank Joshua Zahl for many helpful  suggestions during the preparation of this manuscript. Some AI assistance was also used in figuring out various parts of the proof.

\section{Technical preliminaries and tools}

Since we will be obtaining our result by tweaking the induction architecture of Katz and Zahl \cite{KZ}, we borrow their setup. 

\subsection{Tubes, shadings, and refinements}
\begin{defn}
A $\delta$-tube is the $\delta$-neighborhood of a unit line segment in $\RR^4$. Every $\delta$-tube has measure $\sim\delta^3$. We say that two $\delta$-tubes are essentially identical if the 2-fold dilate of one of the tubes contains the other. If two tubes are not essentially identical then we say they are essentially distinct.
\end{defn}

\begin{defn}
A $\delta$-cube is a set of the form $Q=[0,\delta)^4+v$, where $v\in (\delta\ZZ)^4$. Observe that the set of all $\delta$-cubes tile $\RR^4$. The symbol $Q$ will always refer to a $\delta$-cube, so for example the expression $\sum_{Q\subset A}f(Q)$ will refer to a sum taken over all $\delta$-cubes contained in the set $A$.
\end{defn}

\begin{defn}\label{shading}
Let $T$ be a $\delta$-tube. A shading of $T$ is a set $Y(T)$ that is a union of $\delta$-cubes, each of which intersect $T$. Let $\tubes$ be a set of $\delta$-tubes; for each $T\in\tubes$, let $Y(T)$ be a shading of $T$. We refer to the pair $(\tubes,Y)$ as a set of tubes and their associated shading.
\end{defn}

For each $\delta$-cube $Q$, define 
\begin{equation*}
\tubes_Y(Q)=\{T\in\tubes\colon Q\subset Y(T)\}.
\end{equation*}

If $(\tubes,Y)$ is a set of tubes and their associated shading, define $\mathcal{Q}(Y)$ to be the set of $\delta$-cubes that are contained in at least one shading $Y(T)$ for some $T\in\tubes$. 

If $(\tubes,Y)$ is a set of tubes and their associated shading, define 
\begin{equation*}
\lambda_Y = \frac{1}{|\tubes|}\sum_{T\in\tubes}\frac{|Y(T)|}{|T|},
\end{equation*}
and define
\begin{equation*}
\mu_Y = \frac{1}{|\mathcal{Q}(Y)|}\sum_{Q\in\mathcal{Q}(Y)}|\tubes_Y(Q)|.
\end{equation*}
$\lambda_Y$ is the average shading density of a tube from $(\tubes,Y)$, and $\mu_Y$ is the average multiplicity of a cube from $\mathcal{Q}(Y)$ (i.e. the average number of tubes from $(\tubes,Y)$ whose shading contains $Q$).

\begin{defn}[Refinements]
Let $(\tubes,Y)$ be a set of tubes and their associated shading and let $t>0$. We say that a pair $(\tubes^\prime,Y^\prime)$ is a \emph{$t$-refinement} of $(\tubes,Y)$ if $\tubes^\prime\subset\tubes,$ $Y^\prime(T)\subset Y(T)$ for each $T\in\tubes^\prime$, and
\begin{equation*}
\sum_{T\in\tubes^\prime}|Y^\prime(T)|\geq t\sum_{T\in\tubes}|Y(T)|.
\end{equation*}

For example, if $(\tubes,Y)$ is a set of tubes and their associated shading, then there exists a $\sim|\log\delta|^{-1}$-refinement $(\tubes^\prime,Y^\prime)$ so that $\mu_{Y^\prime}\leq |\tubes^\prime_{Y^\prime}(Q)|\leq2\mu_{Y^\prime}$ for all $Q\in\mathcal{Q}(Y^\prime)$. Since $|\log\delta|^{-1}$-refinements will frequently occur in our proof, sometimes we will abuse notation and simply refer to them as refinements. 
\end{defn}

\subsection{The two-ends reduction}

\begin{defn}[Two-ends condition]\label{twoEndsDef}
Let $(\tubes,Y)$ be a set of $\delta$-tubes and their associated shading. We say that $(\tubes,Y)$ is $(\eps_1,C_1)$-two-ends if for all $T\in\tubes$ and all $\delta\leq r\leq 1$, we have
\begin{equation}\label{nonConcentrationInTube}
|\{Q\colon Q\subset Y(T)\cap B(x,r)\}|\leq r^{\eps_1}C_1\lambda_Y\delta^{-1}\quad\textrm{for all balls}\ B(x,r).
\end{equation}
\end{defn}

For $0\leq d\leq n$, $a>0$, and $0\leq b\leq 1$, define Assertion $\mathbf{TE}(d,a,b)$ to be the following statement:

\begin{twoEndsAss}
Let $\delta>0$ and let $(\tubes,Y)$ be a set of $\delta$-tubes in $\RR^4$ pointing in $\delta$-separated directions and their associated shading. Suppose that $(\tubes,Y)$ is $(\eps_1,100)$-two-ends. Then for each $\eps>0$,
\begin{equation}
\Big| \bigcup_{T\in\tubes}Y(T) \Big|\geq c\lambda_Y^{a}\delta^{4-d+\eps}(\delta^{3}|\tubes|)^b,
\end{equation}
where $c>0$ is a constant that depends only on $d,a,b,\eps,$ and $\eps_1$.
\end{twoEndsAss}

Note that if Assertion $\mathbf{TE}(d,a,b)$ is true, then Assertion $\mathbf{TE}(d,a^\prime,b)$ is also true for all $a^\prime\geq a$. The two-ends reduction says that Assertion $\mathbf{TE}(d,d,b)$ implies a Kakeya maximal function estimate at dimension $d$.
\begin{prop}[The two-ends reduction]\label{twoEndsConditionThm}
Suppose that $\mathbf{TE}(d,d,b)$ is true for some $1\leq d\leq 4$ and $0\leq b\leq 1$. Then for each $\eps>0$, there exists a constant $c_\eps$ so that the following holds. Let $(\tubes,Y)$ be a set of tubes pointing in $\delta$-separated directions and their associated shading. Then
\begin{equation}
\Big| \bigcup_{T\in\tubes}Y(T) \Big|\geq c_{\eps}\lambda_Y^{d}\delta^{4-d+\eps}(\delta^{3}|\tubes|)^b.
\end{equation}
\end{prop}

We refer the interested reader to \cite{KZ} if they wish to see the proof.

\subsection{The robust transversality reduction}

\begin{defn}[Robust transversality]\label{robustTransDef}
Let $(\tubes,Y)$ be a set of tubes and their associated shading. We say that $(\tubes,Y)$ is $(\eps_2,C_2)$-robustly transverse if for all $\delta$-cubes $Q$, all vectors $v$, and all $\delta\leq 1$, we have the following bound on the number of tubes passing through $Q$ and making small angle $r$ with $v$,
for any $r$ with $\delta < r <1$.
\begin{equation}\label{notTooConcentratedAngle}
|\{T\in\tubes_Y(Q)\colon \angle(v(T),v)\leq r\}|\leq r^{\eps_2}C_2\mu_Y.
\end{equation}
\end{defn}

\begin{robustTransAss}
Let $\delta>0$ and let $(\tubes,Y)$ be a set of $\delta$-tubes in $\RR^4$ pointing in $\delta$-separated directions and their associated shading. Suppose that $(\tubes,Y)$ is $(\eps_1,C_1)$-two-ends and $(\eps_2,100)$-robustly transverse. Then for each $\eps>0$,
\begin{equation}\label{twoEndsConditionRequirement}
\Big| \bigcup_{T\in\tubes}Y(T) \Big|\geq c\lambda_Y^{a}C_1^{-C/\eps_1}\delta^{4-d+\eps}(\delta^{3}|\tubes|)^b,
\end{equation}
where $C=C(d,a,b)$ is a constant that depends only on $d$, $a$, and $b$, and $c>0$ is a constant that depends only on $d,a,b,\eps,\eps_1,$ and $\eps_2$ (in particular, both constants are independent of $\delta$).
\end{robustTransAss}

Note that if Assertion $\mathbf{RT}(d,a,b)$ is true, then Assertion $\mathbf{RT}(d,a,b^\prime)$ is also true for all $b^\prime\geq b$. We have the following proposition.

\begin{prop}\label{robustTransReduction}
Suppose that Assertion $\mathbf{RT}(d,a,b)$ is true for some $1\leq d\leq 4$, $a\geq 0$, and $\frac{d-1}{3}\leq b\leq 1$. Then Assertion $\mathbf{TE}(d,a,b)$ is true.
\end{prop}

Again, we refer the interested reader to \cite{KZ} if they wish to see the proof.

\section{Previous Kakeya-type estimates in $\RR^4$}

We recall some results that will be needed in our proof.

\subsection{Wolff's hairbrush estimate}
In \cite{W}, Wolff proved new Kakeya maximal function estimates using a geometric argument called the ``hairbrush'' argument. Here is a consequence of Wolff's hairbrush argument. The formulation presented here is described in \cite{T3}, and for the proof one can look at \cite{T3} or \cite{KZ}.

\begin{thm}\label{wolffThm}
Assertion $\mathbf{RT}(3,2,1/2)$ is true.
\end{thm}

\subsection{\L{}aba-Tao's X-ray estimate}
We recall Theorem 1.2 from \cite{LT}:
\begin{thm}\label{LabaTaoTheorem}
Let $\tubes$ be a set of essentially distinct $\delta$-tubes in $\RR^4$. Suppose that for each vector $v$, there are at most $m$ tubes from $\tubes$ with $\angle(v,v(T))\leq\delta$. Then for each $\eps>0$, there is a constant $C_\eps$ so that
\begin{equation}
\Big\Vert \sum_{T\in\tubes} \chi_T \Big\Vert_{3/2}\leq C_{\eps} \delta^{-\frac{1}{3}-\eps}m^{\frac{5}{36}}(\delta^3|\tubes|)^{\frac{7}{9}}.
\end{equation}
\end{thm}

It will be more convenient for us to phrase this result as a multiplicity bound on unions of tubes.

\begin{cor}\label{KTCor}
Let $(\tubes,Y)$ be a set of essentially distinct $\delta$-tubes and their associated shading, and let $\Omega\subset S^3$ be a $\delta$-separated set. Define $m = |\tubes|/|\Omega|$. Suppose that each tube from $\tubes$ points in a direction from $\Omega$, and $\sim m$ tubes from $\tubes$ point in each direction from $\Omega$. Suppose $|\tubes_Y(Q)|\leq 10\mu_Y$ for all $Q\in\mathcal{Q}(Y)$. Then 
\begin{equation}\label{muYBound}
\mu_Y \leq C_{\eps} \lambda^{-2} \delta^{-1-\eps}m^{3/4}(\delta^3|\Omega|)^{1/3}.
\end{equation}
\end{cor}

\section{Applying Wang-Zahl inside a polynomial grain}

In this section, we will be working with tubes which are contained in the $\delta$-neighborhood of a low degree variety in $\RR^4$ (also called a grain, see Definition \ref{defgr}). After various technical manipulations, we will apply the 3-dimensional Wang-Zahl result \cite{WZ3} to this configuration of tubes. This is the key new ingredient in the proof, eventually yielding the improved dimension bound.

\subsection{Preliminaries}
A set of tubes is said to satisfy the Wolff axioms if not too many tubes from the set can concentrate into the thickened neighborhood of an affine subspace. More precisely, we have the following.
\begin{defn}
    Let $\tubes$ be a collection of $\delta$-tubes. We say that it satisfies the linear Wolff axioms if at most $Ct_1 t_2\ldots t_{n-1}\delta^{1-n}$ tubes are contained in a rectangular prism of dimensions $1\times t_1 \times t_2 \times \ldots t_{n-1}$.
\end{defn}

A set of tubes is said to satisfy the polynomial Wolff axioms if not too many tubes from the set can concentrate into the thickened neighborhood of an algebraic variety, or more generally a semi-algebraic set. This will play a crucial role in reaching our objective.

\begin{defn}
    Let $\tubes$ be a collection of $\delta$-tubes. We say that it satisfies the polynomial Wolff axioms if for every semialgebraic set $S\subset \RR^n$ of complexity at most $E$, and every $\delta\leq \lambda \leq 1$,
    \begin{equation}
        |\{ T\in \tubes : |T\cap S|\geq \lambda |T| \} | \leq K_E |S| \delta^{1-n} \lambda^{-n}.
    \end{equation}
\end{defn}

In \cite{Z} and \cite{KR}, Zahl, and independently Katz and Rogers showed that collections of direction-separated tubes satisfy these requirements for $n=4$. Here is the formal statement.

\begin{thm}\label{katzrog}
    Let $n, E\geq 2$ be integers and $\epsilon>0$. Then there is a constant $C(n,E,\epsilon)>0$ so that, for every set $\tubes$ of $\delta$-tubes in $\RR^n$, pointing in $\delta$-separated directions.
    \begin{equation}
        \# (\{T\in \tubes: |T\cap S|\geq \lambda |T|\})\leq C(n,E,\epsilon)|S|\delta^{1-n-\epsilon}\lambda^{-n}
  \end{equation}
  whenever $S$ is a semi-algebraic set of complexity at most $E$, and $\lambda\geq \delta>0$.
\end{thm}

Next we state a special case of the theorem of Wang and Zahl \cite {WZ3} which we will be using.
\begin{thm}\label{wangzahlthm}
    Let $\tubes$ be a collection of $\delta$-tubes in $\RR^3$ which satisfies the linear Wolff axioms. Then we get
    \begin{equation}
         \Bigg| \bigcup_{T\in \tubes} T \Bigg|\geq C_{\epsilon} \delta^\epsilon (\delta^2|\tubes|)
    \end{equation}
\end{thm}

\subsection{The result}

Here we state the proposition that we get by applying Wang-Zahl inside a polynomial grain.

\begin{prop}\label{wangzahlgrain}
    Let $P\in \RR[x_1,x_2,x_3,x_4]$ be a non-zero polynomial of degree at most $d$, and let $Z=Z(P)\cap [-1,1]^4$. Let $\tubes$ be a collection of $\delta$-tubes with $T\subset N_\delta(Z)$ for all $T\in \tubes$. Further suppose that the tubes of $\tubes$ point in $\delta$-separated directions. Then we have
    \begin{equation}
        \Bigg| \bigcup_{T\in \tubes} T \Bigg|\geq C_{\epsilon,d} \delta^\epsilon (\delta^3|\tubes|).
    \end{equation}
\end{prop}

\subsection{Reduction to a graphical decomposition lemma}

Working with the variety $Z(P)$ directly is complicated, so we break it up into more regular pieces and then work with those. The following definition will help us build these regular pieces.

\begin{defn}
    A semi-algebraic subset $M$ of $\RR^k$ is a smooth semi-algebraic manifold of dimension $l$ if for every $x\in M$, there exists a semi-algebraic open set $U$ of $\RR^k$ and a smooth semi-algebraic diffeomorphism $\varphi$ from $U$ to a semi-algebraic open neighborhood $\Omega$ of $x\in \RR^k$ such that $\varphi(0)=x$ and 
    \begin{equation*}
        \varphi(U\cap (\RR^l\times \{0\}))=M\cap \Omega
    \end{equation*}
    (where $\RR^l\times \{0\}=\{(a_1,a_2,\ldots, a_l)\in \RR^l\}).$  
\end{defn}

See \cite{BPR, BCR} for more background. With this definition, we are ready to state our decomposition lemma. We will use the notation $e_1=(1,0,0,0)$, $e_2=(0,1,0,0)$, $e_3=(0,0,1,0)$ and $e_4=(0,0,0,1)$. The statement can be seen as a variant of Proposition 3.1 from \cite{Z2}, tailored to our purpose.
    
\begin{lem}\label{grdecomp} (Graphical decomposition)
    Let $P\in \RR[x_1,x_2,x_3,x_4]$ be a non-zero polynomial of degree at most $d$, and let $Z=Z(P)\cap [-1,1]^4$. Then there exists $N, C\lesssim_d 1$, and $v_n \in \{e_1,e_2,e_3,e_4\}$ for $1\leq n \leq N$, and connected smooth semi-algebraic manifolds $S_n\subset Z$; such that
    \begin{enumerate}[label=(\alph*)]
        \item we have $Z=\bigcup_{n=1}^N S_n$.
        \item each $S_n$ has complexity at most $C$.
        \item for all $p\in S_n$, we have $\operatorname{proj}_{N_p S_n}(v_n)\geq 1/2$.
        \item if $\pi_n:\RR^4\to v_n^\perp$ denotes the orthogonal projection, then $S_n$ is a graph in the $v_n$ direction of a function with Lipschitz constant at most $C$.
        \item the set $S_{n,\delta}=N_\delta(S_n)$ is semi-algebraic of complexity at most $C$ (uniformly in $\delta$), and
        \begin{equation}
            \mathcal{H}^1 (S_{n,\delta}\cap \pi_n^{-1}(y))\leq C\delta
        \end{equation}
        for all $y\in v_n^\perp$.
    \end{enumerate}
\end{lem}

This lemma gives us the technical support needed to proceed.

\begin{proof}[Proof of Proposition \ref{wangzahlgrain} assuming Lemma \ref{grdecomp}]

By Lemma \ref{grdecomp}, we write $Z=\bigcup_{n=1}^N S_n$ with all the properties stated. We also have
\begin{equation}
    N_\delta (Z)\subset \bigcup_{n=1}^N N_\delta (S_n).
\end{equation}
Therefore, each $T\in \tubes$ satisfies
\begin{equation}
    T\subset \bigcup_{n=1}^N N_\delta (S_n) \implies \sum_{n=1}^N |T\cap N_\delta (S_n)|\geq |T|.
\end{equation}
So there exists $n(T)$ such that
\begin{equation}
    |T\cap N_\delta (S_{n(T)})|\geq N^{-1}|T|.
\end{equation}
Pigeonholing in $n(T)$, there exists $\tubes_0\subset \tubes$ and $n_0$ such that
\begin{equation}
    |\tubes_0|\geq N^{-1}|\tubes| \quad \text{and} \quad |T\cap N_\delta (S_{n_0})|\geq N^{-1}|T|.
\end{equation}
We restrict our attention to the tubes in $\tubes_0$. This is harmless since $|\tubes_0|\gtrsim_d |\tubes|$ as we just saw. Furthermore, we relabel $T\in \tubes_0$ to mean $T\cap N_\delta (S_{n_0})$. This is a slight abuse of notation, since $T\cap N_\delta (S_{n_0})$ has $O_d(1)$ many connected components (see \cite{BPR, BCR}) and we mean to pick one. Again, this is acceptable since $|T\cap N_\delta (S_{n_0})|\gtrsim_d |T|$. We now have a collection of tubes $\tubes_0$ contained in the $\delta$-neighborhood of a smooth semi-algebraic manifold $S_{n_0}$ of complexity at most $C\lesssim_d 1$. Without loss of generality, we may assume that $v_{n_0}=e_4$. Therefore (c) tells us the normal in roughly vertical, (d) tells us that the manifold is a graph of a Lipschitz function with constant $C\lesssim_d 1$ and (e) tells us the intersection of $N_\delta (S_{n_0})$ with any vertical line is $\lesssim_d \delta$.

Now let $\pi_4(x_1,x_2,x_3,x_4)=(x_1,x_2,x_3)$. Then let $\tubes_0'$ be the projected set of tubes in $[-1,1]^3$ (this is a legal move because of (c)). Now, we claim that that this collection satisfies the linear Wolff axioms. Take a $1\times t_1\times t_2$ rectangular prism $R$. The number of tubes of $\tubes_0'$ inside $R$ is given by the number of tubes of $\tubes_0$ inside $\pi_4^{-1}(R)\cap N_\delta (S_{n_0})$. This is a semi-algebraic set of complexity $\lesssim_d 1$. Every vertical line in the cube $[-1,1]^4$ intersects $N_\delta(S_{n_0})$ in an interval of length $\lesssim_d\delta$ by (e). Therefore the volume of this semi-algebraic set is $\lesssim t_1t_2\delta$. Therefore, by applying Theorem \ref{katzrog}, we get the number of tubes contained is $\lesssim (t_1t_2\delta) \delta^{-3}=t_1t_2 \delta^{-2}$. This establishes the linear Wolff axioms claim. 

Finally, by applying the Wang-Zahl result (Theorem \ref{wangzahlthm}), we get
    \begin{equation}
        \bigg|\bigcup_{T' \in \tubes_0'} T' \bigg|\gtrapprox \delta^2 |\tubes_0'|.
    \end{equation}
    Again, retracing the fiber, we get a factor of $\delta$ which completes the proof.
\end{proof}

\subsection{Proof of graphical decomposition lemma}

We will now prove the graphical decomposition lemma.

\begin{proof}[Proof of Lemma \ref{grdecomp}]
We can write 
\begin{equation}\label{stratif}
    Z=\bigsqcup_{\alpha=1}^A Z_\alpha,
\end{equation}
where each $Z_\alpha$ is a connected smooth semi-algebraic manifold of dimension $d_\alpha\leq 3$ and $A\lesssim_d 1$. This follows from Theorem 5.38 of \cite{BPR}, or from Proposition 2.9.10 of \cite{BCR}. It can be seen as a form of cylindrical decomposition of semi-algebraic sets and should be compared with Theorem 3.2 of \cite{Z2}.

Fix a manifold $M=Z_\alpha$. For $p\in M$, let $T_p M$ be the tangent space to $M$ at $p$. Correspondingly, we have the normal space
\begin{equation}
    N_p M=(T_p M)^\perp.
\end{equation}
Now, if $\operatorname{proj}_{N_p M}$ denotes the orthogonal projection, then
\begin{equation}
    \sum_{j=1}^4 |\operatorname{proj}_{N_p M}(e_j)|^2=\operatorname{trace} (\operatorname{proj}_{N_p M})=4-d_\alpha\geq 1.
\end{equation}
So there exists $j$ such that $|\operatorname{proj}_{N_p M}(e_j)|\geq 1/2$. We define
\begin{equation}\label{normalvertical}
    M_j=\{p\in M: |\operatorname{proj}_{N_p M}(e_j)|\geq 1/2\}.
\end{equation}
Then we get that $M=M_1\cup M_2\cup M_3\cup M_4$. These are semi-algebraic sets of bounded complexity. Therefore we obtain a refinement of \eqref{stratif} with the the added condition that for each $\alpha$, there exists $j$ such that $|\operatorname{proj}_{N_p Z_\alpha}(e_j)|\geq 1/2$ for all $p\in Z_\alpha$. We may also assume that for this $j$, the semi-algebraic manifold $Z_\alpha$ is the graph of a function in the $e_j$ direction. At this point (a), (b) and (c) have been established, we now work towards establishing (d) and (e).

Again, fix an $\alpha$ (and corresponding $j$), and let $M=Z_\alpha$. We have the projection map $\pi_j:\RR^4\to e_j^\perp$. Let us now restrict this map to $T_p M$. For $v\in T_p M$, we have
\begin{equation}\label{projdotprod}
    |\pi_j(v)|^2=|v|^2-|(v,e_j)|^2.
\end{equation}
Also, by \eqref{normalvertical}, we get that
\begin{equation}
    |\operatorname{proj}_{T_p M} (e_j)|^2=1-|\operatorname{proj}_{N_p M} (e_j)|^2\leq \frac{3}{4}.
\end{equation}
In particular, we have $|\operatorname{proj}_{T_p M} (e_j)|\leq \sqrt{3}/2$. This implies
\begin{equation}
    |(v,e_j)|\leq \frac{\sqrt{3}}{2} |v|.
\end{equation}
Plugging this into \eqref{projdotprod}, we get that
\begin{equation}
    |\pi_j(v)|\geq \frac{1}{2}|v|.
\end{equation}
Therefore, if $f$ is the function corresponding to $M$, we have $|Df|\leq 3$. Now by applying Proposition 3.6 from \cite{Z2}, we get the Lipschitz property (d), and (e) follows from (d). 
\end{proof}

\section{New variant of Guth-Zahl's trilinear Kakeya estimate}

By applying Proposition \ref{wangzahlgrain} as a new ingredient, we obtain a slight improvement to the trilinear estimate (Proposition 3.1) of Guth and Zahl \cite{GZ}, namely that that the exponent of the wedge product changes from $12/13$ to $9/13$.

\begin{thm}[Trilinear Kakeya in $\RR^4$]\label{GZTrilinearProp}
For each $\epsilon>0$, there exists a constant $C_\epsilon$ so that the following holds. Let $\tubes$ be a set of $\delta$-tubes in $\RR^4$ that point in $\delta$-separated directions. Then
\begin{equation}\label{trilinearBoundFromGZ}
\int\Big(\sum_{T_1,T_2,T_3\in\tubes}\chi_{T_1}\ \chi_{T_2}\ \chi_{T_3}\ |v_1\wedge v_2\wedge v_3|^{9/13}\Big)^{13/27} \leq C_{\epsilon}\delta^{-1/3-\epsilon}(\delta^3|\tubes|)^{4/3},
\end{equation}
where in the above expression $v_i$ is the direction of the tube $T_i$.
\end{thm}

Correspondingly, we get an improved corollary. Note that the exponent of theta changes from 1 to 3/4 when you compare with Corollary 20 in \cite{KZ}. The proof of the corollary follows the same route.

\begin{cor}\label{trilinearVolumeBound}
For each $\epsilon>0$, there exists a constant $C_\epsilon$ so that the following holds. Let $(\tubes,Y)$ be a set of $\delta$-tubes in $\RR^4$ that point in $\delta$-separated directions and their associated shading. Suppose that for each $Q\in\mathcal{Q}_Y(T)$, we have 
\begin{equation*}
|\{(T_1,T_2,T_3)\in\tubes_Y(Q)\colon |v_1\wedge v_2\wedge v_3|\geq\theta\}|\geq s|\tubes_Y(Q)|^3.
\end{equation*}
Then
\begin{equation}\label{volumeBoundTrilinear}
\Big|\bigcup_{T\in\tubes}Y(T)\Big|\geq c_{\eps}\delta^{\eps} s^{9/4} \lambda_Y^{13/4}\delta^{3/4}\theta^{3/4}(\delta^3|\tubes|)^{1/4}.
\end{equation}
\end{cor}

Equivalently, we get a pointwise bound
\begin{equation}\label{trilinearPointwise}
|\tubes_Y(Q)|\lessapprox \lambda^{-9/4}s^{-9/4}\theta^{-3/4}\delta^{-3/4}(\delta^3|\tubes|)^{3/4}\quad\textrm{for all}\ Q\in\mathcal{Q}(Y).
\end{equation}

\begin{proof}
After refining $(\tubes,Y)$, we can suppose that $|\tubes_Y(Q)|\sim\mu_Y$ for each $Q\in\mathcal{Q}_Y(T)$. Then
\begin{equation}
\begin{split}
&\delta^4|\mathcal{Q}(Y)|\big(\mu_Y^3\theta^{9/13}\big)^{13/27}\\
&\sim \sum_{Q\in\mathcal{Q}_Y(T)}|Q|\ \big(|\tubes_Y(Q)|^3\ \theta^{9/13}\big)^{13/27}\\
&\sim \sum_{Q\in\mathcal{Q}_Y(T)}|Q|\ s^{-1}\big(|\{T_1,T_2,T_3\in\tubes_Y(Q)\colon |v_1\wedge v_2\wedge v_3|\sim\theta\}|\ \theta^{9/13}\big)^{13/27}\\
&\leq s^{-1}\sum_{Q\in\mathcal{Q}_Y(T)}|Q| \big(\sum_{T_1,T_2,T_3\in\tubes_Y(Q)} |v_1\wedge v_2\wedge v_3|^{9/13}\big)^{13/27}\\
&\leq s^{-1}\int\Big(\sum_{T_1,T_2,T_3\in\tubes}\chi_{T_1}\ \chi_{T_2}\ \chi_{T_3}\ |v_1\wedge v_2\wedge v_3|^{9/13}\Big)^{13/27}\\
&\leq C_{\epsilon}s^{-1}\delta^{-1/3-\epsilon}(\delta^3|\tubes|)^{4/3}.
\end{split}
\end{equation}
Rearranging, we get
\begin{equation*}
\mu_Y^{13}\leq C_{\eps}^9 s^{-9} \delta^{-3-9\eps}\theta^{-3}(\delta^4|\mathcal{Q}(Y)|)^{-9}(\delta^3|\tubes|)^{12}.
\end{equation*}
Since $\mu_Y|\bigcup_{T\in\tubes}Y(T)|\approx \lambda_Y (\delta^3|\tubes|)$ and $\delta^4|\mathcal{Q}(Y)|\sim |\bigcup_{T\in\tubes}Y(T)|$, we have
\begin{equation*}
\lambda_Y^{-13}(\delta^3|\tubes|)^{13}|\bigcup_{T\in\tubes}Y(T)|^{-13}\leq C_{\eps}^9s^{-9}\delta^{-3-9\eps}\theta^{-3}|\bigcup_{T\in\tubes}Y(T)|^{-9}(\delta^3|\tubes|)^{12}.
\end{equation*}
Re-arranging, we obtain \eqref{volumeBoundTrilinear}.
\end{proof}

\subsection{Strategy for the proof of Theorem \ref{GZTrilinearProp}}

The structure of the proof will be very similar to the proof of Proposition 3.1 in Guth-Zahl \cite{GZ}. We will be borrowing heavily from there, and applying the new Proposition \ref{wangzahlgrain} in the middle to get the improvement. We now state some tools which we will be using.
\subsubsection{Multilinear Kakeya}
\begin{thm}[Bourgain-Guth, \cite{BG}, Theorem 6]\label{BourgainGuthTHhm}
Let $\tubes_1,\tubes_2,\tubes_3$ be three sets of $\delta$--tubes in $\RR^4$. Then
\begin{equation}
\Big(\int \sum_{T_1,T_2,T_3\in\tubes_1\times\tubes_2\times\tubes_3}\chi_{T_1}\ \chi_{T_2}\ \chi_{T_3}\ v_1\wedge v_2\wedge v_3 \Big)^{1/2}\lesssim \delta^4 (|\tubes_1|\ |\tubes_2|\ |\tubes_3|)^{1/2}.
\end{equation}
\end{thm}

\begin{rem} Theorem 6 from \cite{BG} states this result for the case $|\tubes_1|=|\tubes_2|=|\tubes_3|$, but the proof gives the estimate above when the three cardinalities are not equal. \end{rem}
\subsubsection{Grains decomposition}

This was developed in Guth-Zahl \cite{GZ}. We refer the reader there for more details/expanded discussion.

\begin{defn}[Grain]\label{defgr}
A grain (of complexity $d$) in $\RR^n$ is the $C\delta$--neighborhood of a semi-algebraic set of dimension $\leq n-1$ and complexity $\leq d$.
\end{defn}

\begin{defn}[Grains decomposition]
Let $\mathcal{Q}$ be a set of cubes in $\RR^n$. A \emph{Grains decomposition} of $\mathcal{Q}$ of degree $d$ and error $\epsilon$ is a set of grains $\mathcal{G}$, each of which has complexity $\leq d$. This decomposition has the following properties.
\begin{itemize}
\item The cubes are evenly-distributed across the grains: For each $G\in\mathcal{G}$, there is a set $\mathcal{Q}_G\subset \mathcal{Q}$. The sets $\{\mathcal{Q}_G\}$ are disjoint; if $Q\in\mathcal{Q}_G$ then $Q\subset G$. We have 
\begin{equation}\label{mostCubesCaptured}
\sum_{G\in\mathcal{Q}}|\mathcal{Q}_G|\geq\delta^\epsilon|\mathcal{Q}|, 
\end{equation}
and
\begin{equation}\label{grainsSameSize}
\delta^{\epsilon}|\mathcal{Q}|/|\mathcal{G}|\leq|\mathcal{Q}_G|\leq\delta^{-\epsilon}|\mathcal{Q}|/|\mathcal{G}|.
\end{equation}

\item A tube doesn't intersect too many grains: If $T$ is a $\delta$--tube, then
\begin{equation}\label{notTooManyTubesPerGrain}
|\{G\in\mathcal{G}\colon T\cap Q\neq\emptyset\ \textrm{for some}\ Q\in\mathcal{Q}_G\}|\leq C_\epsilon\delta^{-\epsilon} |\mathcal{G}|^{1/n}.
\end{equation}
\end{itemize}
\end{defn}

\begin{prop}\label{existsGrainsDecomp}
Let $\mathcal{Q}$ be a set of $\delta$-cubes in $\RR^n$. Then for each $\epsilon>0$, there exists a grains decomposition of $\mathcal{Q}$ of degree $d(\epsilon)$ and error $\epsilon$.
\end{prop}

This proposition is proved by polynomial partitioning (see \cite{GZ} for proof). We will only use Proposition \ref{existsGrainsDecomp} in the special case $n=4$.

\subsection{Proof of Theorem \ref{GZTrilinearProp}}

\begin{proof}[Proof of Theorem \ref{GZTrilinearProp}]
Fix $\epsilon>0$. Let $d(\cdot)$ be the function from Proposition \ref{existsGrainsDecomp}. 

For simplicity of notation, we will replace each tube $T\in\tubes$ with the union of all $\delta$--cubes that intersect $T$; thus we should think of each $T\in\tubes$ as a union of cubes. 

Now we do some dyadic pigeonholing.  First we dyadic pigeonhole $|v_1 \wedge v_2 \wedge v_3|$.  There is a dyadic $\theta$ so that

$$\int\Big(\sum_{T_1,T_2,T_3\in\tubes}\chi_{T_1}\ \chi_{T_2}\ \chi_{T_3}\ |v_1\wedge v_2\wedge v_3|^{9/13}\Big)^{13/27} \sim \int\Big(\sum_{\substack{T_1,T_2,T_3\in\tubes\\ v_1\wedge v_2\wedge v_3 \sim\theta}
} \chi_{T_1}\ \chi_{T_2}\ \chi_{T_3}\ |v_1\wedge v_2\wedge v_3|^{9/13}\Big)^{13/27}. $$

Next we divide $B(0,2)$ into $\delta$-cubes $Q$ and we dyadic pigeonhole their contribution to the last integral.  We can choose a set of cubes $\mathcal{Q}$ with $\cup_{Q \in \mathcal{Q}} Q = A$ so that

$$\int_{B(0,2)} \Big(\sum_{T_1,T_2,T_3\in\tubes}\chi_{T_1}\ \chi_{T_2}\ \chi_{T_3}\ |v_1\wedge v_2\wedge v_3|^{9/13}\Big)^{13/27} \sim \int_A \Big(\sum_{\substack{T_1,T_2,T_3\in\tubes\\ v_1\wedge v_2\wedge v_3 \sim\theta}
} \chi_{T_1}\ \chi_{T_2}\ \chi_{T_3}\ |v_1\wedge v_2\wedge v_3|^{9/13}\Big)^{13/27}, $$

\noindent and so that each $Q \in \mathcal{Q}$ makes a roughly equal contribution to the right-hand side.  

Let $\mathcal{G}$ be a grains decomposition of $\mathcal{Q}$ with error $\epsilon$ and degree $d=d(\epsilon)$.  Next we consider how much a tube intersects a grain.  The intersection of a tube with a grain could have a few different connected components, perhaps with different lengths.  Since the grain is a semi-algebraic set of complexity at most $d$, the number of these components is $O_d(1)$, and there is no harm in treating them separately.  If $W\subset\RR^4$, then we write $\operatorname{CC}(W,x)$ for the Euclidean connected component of $W$ containing $x$.  If $G$ is a grain and $T$ is a tube, then note that $\operatorname{diam}(\operatorname{CC}(T\cap G,x ))$ is the length of the component of $T \cap G$ thru the point $x$.  Next we dyadic pigeonhole these lengths:  we can find $\ell_1,\ell_2,\ell_3$ so that
\begin{equation}\label{pigeonholeEllI}
(\operatorname{LHS}\ \eqref{trilinearBoundFromGZ})\lessapprox\sum_{G\in\mathcal{G}}\int_{G\cap A} \Big(\sum_{\substack{T_1,T_2,T_3\in\tubes\\v_1\wedge v_2\wedge v_3 \sim\theta\\ \ell_i\leq \operatorname{diam}(\operatorname{CC}(T_i\cap G,x ))\leq 2\ell_i } } \theta^{9/13}\ \chi_{T_1}\ \chi_{T_2}\ \chi_{T_3}\Big)^{13/27}.
\end{equation} 

Without loss of generality, we can assume that $\ell_1=\max(\ell_1,\ell_2,\ell_3)$. Cover each grain $G$ by balls of radius $C\ell_1$ so that the balls are $O(1)$ overlapping, and any subset of $G$ of diameter $\leq 2\ell_1$ is entirely contained within one of the balls. The intersection of $G$ with a ball of this type will be called a sub-grain $G^\prime$ with parent $G$. If $G^\prime$ is a sub-grain of $G$, let $\mathcal{Q}_{G^\prime}=\{Q\in\mathcal{Q}_G\colon Q\subset G^\prime\}$. Let $\mathcal{G}^\prime$ denote the set of sub-grains.

Note that if $T$ is a $\delta$--tube and $G\in\mathcal{G}$ is a grain satisfying $T\cap G\neq\emptyset$, then there are $O_d(1)$ sub-grains $G^\prime\in\mathcal{G}^\prime$ with $G^\prime\subset G$ that contain a point $x\in T\cap G^\prime$ with $\operatorname{diam}(\operatorname{CC}(T_i\cap G),x)\leq 2\ell_i$. This is because $T\cap G$ is a semi-algebraic set of bounded-complexity, so it is a union of $O_d(1)$ connected sets (see \cite{BCR} for further information and background on semi-algebraic sets). Each of these sets with diameter $\leq 2\ell_i$ intersects at most $O(1)$ sub-grains $G^\prime$. Furthermore, if there exists $x\in T\cap G$ with $\ell_i\leq \operatorname{diam}(\operatorname{CC}(T_i\cap G),x)\leq 2\ell_i$, then there exists at least one sub-grain $G^\prime$ with $c_d\ell_i\leq \operatorname{diam}(\operatorname{CC}(T_i\cap G^\prime),x)\leq 2\ell_i$, where $c_d>0$ depends only on $d$. For each sub-grain $G^\prime$, define
$$
\tubes_{i,G^\prime}=\{T\cap G^\prime\colon\ \textrm{there exists a component}\ W\subset T\cap G\ \textrm{with}\ W\subset G^\prime,\ c_d\ell_i\leq \operatorname{diam}(W)\leq 2\ell_i\}.
$$

Thus for each $i=1,2,3$,
\begin{equation}\label{smallGrainsComparableLarge}
\sum_{G\in\mathcal{G}}|\{ T\in\tubes\colon \textrm{there exists a component}\ W\subset T\cap G\ \textrm{with}\ \ell_i\leq \operatorname{diam}(W)\leq 2\ell_i\}|\sim \sum_{G^\prime\in\mathcal{G}^\prime}|\tubes_{i,G^\prime}|.
\end{equation}

After dyadic pigeonholing and refining the grains in $\mathcal{G}^\prime$, we can assume that \eqref{mostCubesCaptured}, \eqref{grainsSameSize}, and \eqref{notTooManyTubesPerGrain} hold for $\mathcal{G}^\prime$ (doing this may have made $|\mathcal{G}^\prime|$ smaller by a factor of $\lessapprox 1$). Furthermore, since each grain $G^\prime\in\mathcal{G}^\prime$ is contained in a ball of radius $\ell_1$, we have 
\begin{equation}\label{grainSize}
|G^\prime|\lesssim \ell_1^3\delta
\end{equation}
for all $G^\prime\in\mathcal{G}^\prime$, where the implicit constant depends only on $d(\epsilon),$ which in turn depends only on $\epsilon$ (see e.g.~\cite{Won}).

After further dyadic pigeonholing the grains $G\in\mathcal{G}^\prime$, we can find numbers $N_1,N_2,N_3$ so that $N_i\leq |\tubes_{i,G^{\prime}}|\leq 2N_i$ for each $G^{\prime}\in\mathcal{G}^{\prime}$.  Since $\mathcal{G}^\prime$ obeys \eqref{notTooManyTubesPerGrain},
\begin{equation*}
N_i|\mathcal{G}^{\prime}|\lessapprox C_\epsilon\delta^{-\epsilon} |\tubes|\ |\mathcal{G}^{\prime}|^{1/4},
\end{equation*}
i.e.
\begin{equation}\label{boundOnNi}
N_i \lessapprox  C_\epsilon\delta^{-\epsilon}|\tubes|\ |\mathcal{G}^{\prime}|^{-3/4},\qquad i=1,2,3.
\end{equation}

Let $\mathcal{Q}^{\prime}=\bigcup_{\mathcal{G}^{\prime}}\mathcal{Q}_G$. After dyadic pigeonholing, we can find numbers $\mu$ and $\mu_1,\mu_2,\mu_3$ so that if we refine $\mathcal{Q}^{\prime}$ and the associated sets $\mathcal{Q}_{G^\prime}$, then if $Q\in\cup\mathcal{Q}_{G^\prime}$ for some $G^\prime\in\mathcal{G}^{\prime}$ and if $x\in Q$, then
$$
\sum_{\substack{(T_1,T_2,T_3)\in\tubes_{1,G^\prime}\times \tubes_{2,G^\prime}\times \tubes_{3,G^\prime}\\v_1\wedge v_2\wedge v_3 \sim\theta} }\chi_{T_1}(x)\chi_{T_2}(x)\chi_{T_3}(x)\sim \mu^3,
$$
and $\sim\mu_i$ tubes from $\tubes_i$ pass through $x$ for each $i=1,2,3$. Note that
\begin{equation}\label{muPrimeVsMu123}
\mu \leq(\mu_1\mu_2\mu_3)^{1/3}. 
\end{equation}
\begin{rem}
The LHS of \eqref{muPrimeVsMu123} might be much smaller than the RHS if (for example) the main contribution to \eqref{trilinearBoundFromGZ} from a typical point comes from triples with $v_1\wedge v_2\wedge v_3\sim 1$, but at the same time most triples through a typical point have $v_1\wedge v_2\wedge v_3$ much smaller than 1.
\end{rem}

After a further refinement of $\mathcal{G}^{\prime}$ to throw away those grains for which $|\mathcal{Q}_G|$ is small, we can assume that $|\mathcal{Q}_G|\approx |\mathcal{Q}^{\prime}|/|\mathcal{G}^{\prime}|$ for all $G\in\mathcal{G}^{\prime}$. 

Let $A^{\prime}=\bigcup_{\mathcal{Q}^{\prime}}Q$. Then $|A^{\prime}|\approx |A|$.  Because of all the dyadic pigeonholing, we have

\begin{equation} \label{simpLHS} (\operatorname{LHS}\ \eqref{trilinearBoundFromGZ}) \approx |A^{\prime}| \mu^{13/9} \theta^{1/3}. \end{equation}

To prove \eqref{trilinearBoundFromGZ}, it suffices to show that

\begin{equation} \label{trilineargoal} |A^\prime|^9 \mu^{13} \theta^3 \lesssim  \delta^{-3 - 9 \eps} (\delta^3 | \tubes| )^{12}. \end{equation}

Let $i_0$ be the index so that $\mu_{i_0}=\max(\mu_1,\mu_2,\mu_3)$. So in particular, \eqref{muPrimeVsMu123} implies
\begin{equation}\label{muLeqMuI0}
\mu \leq\mu_{i_0}.
\end{equation}

We note that 
\begin{equation}
\begin{split}
\mu_i|A^{\prime}| &\lesssim \sum_{G^\prime\in\mathcal{G}^{\prime}} \sum_{T\in\tubes_{i,G}}|T\cap A^{\prime}\cap G^\prime|\\
&\leq \sum_{G^\prime\in\mathcal{G}^{\prime}} \sum_{T\in\tubes_{i,G^\prime}}|T\cap G^\prime|\\
&\leq  |\mathcal{G}^{\prime}|N_i\ell_i\delta^3\\
&\lessapprox  C_\epsilon\delta^{-\epsilon}|\mathcal{G}^{\prime}|^{1/4}\ell_i(\delta^3|\tubes|),
\end{split}
\end{equation}
where on the last line we used \eqref{boundOnNi}. Thus
$$
\ell_i\gtrapprox  c_\epsilon\delta^{\epsilon}\mu_i\ |A^{\prime}|\ |\mathcal{G}^{\prime}|^{-1/4}(\delta^3|\tubes|)^{-1}.
$$
In particular, since $\ell_1=\max(\ell_1,\ell_2,\ell_3)$, we have
\begin{equation}\label{ell1MuI0}
\ell_1\gtrapprox  c_\epsilon\delta^{\epsilon}\mu_{i_0}|A^\prime|\ |\mathcal{G}^{\prime}|^{-1/4}(\delta^3|\tubes|)^{-1}.
\end{equation}

Everything up until this point was the same as in the proof of Proposition 3.1 in \cite{GZ}. Over here the proof diverges. Instead of applying polynomial Wolff axioms to get an upper bound on $N_1$, we use Wang-Zahl \cite{WZ3} inside a grain to do the same. Polynomial Wolff axioms of course show up in the proof of applying Wang-Zahl inside a grain (namely Proposition \ref{wangzahlgrain}), so it is baked inside here as well.

Let $G^\prime\in\mathcal{G}^\prime$. By applying Proposition \ref{wangzahlgrain}, we get that
\begin{equation}
\frac{|A^\prime|}{|\mathcal{G}^\prime|}\gtrapprox \delta^3 N_1 l_1^4.
\end{equation}
Rearranging and applying \eqref{ell1MuI0}, we get
\begin{equation}\label{gamechanger}
N_1\lessapprox \frac{|A^\prime|}{|\mathcal{G}^\prime|\delta^3 l_1^4}\lessapprox \delta^{-3}\mu_{i_0}^{-4}|A^\prime|^{-3}(\delta^{3}|\tubes|)^4.
\end{equation}

Furthermore,

\begin{equation}
\begin{split}
|A^\prime|\ |\mathcal{G}^\prime|^{-1}\mu^{3/2}\theta^{1/2}&\approx\int_{G^\prime\cap A^\prime} \Big(\sum_{\substack{T_1,T_2,T_3\in\tubes_{1,G^\prime}\times\tubes_{2,G^\prime}\times\tubes_{3,G^\prime}\\v_1\wedge v_2\wedge v_3 \sim\theta\\ } }\theta \chi_{T_1}\chi_{T_2}\chi_{T_3}\Big)^{1/2}\\
&\lesssim \delta^4(N_1N_2N_3)^{1/2}\\
&\lessapprox \delta(\delta^3|\tubes|) |\mathcal{G}^\prime|^{-3/4} N_1^{1/3}N_1^{1/6}\\
&\lessapprox \delta (\delta^3|\tubes|)|\mathcal{G}^\prime|^{-3/4} ( |\tubes|\ |\mathcal{G}^\prime|^{-3/4})^{1/3}(\delta^{-3}\mu_{i_0}^{-4}|A^\prime|^{-3}(\delta^{3}|\tubes|)^4)\delta^{-2})^{1/6}\\
&\lessapprox \delta^{-1/2}|\mathcal{G}^\prime|^{-1}\mu_{i_0}^{-4/6}|A^\prime|^{-1/2}(\delta^3|\tubes|)^{2},
\end{split}
\end{equation}
where on the third line we used Theorem \ref{BourgainGuthTHhm}, on the fourth line we used \eqref{boundOnNi}, and on the fifth line we used \eqref{boundOnNi} and \eqref{gamechanger}. Instead of splitting the exponent of $N_1$ in the form 3/8+1/8=1/2 as in Guth-Zahl \cite{GZ}, splitting it as 1/6+1/3=1/2 the way we have done above is what works for the numerology we have at hand.

Re-arranging and using \eqref{muLeqMuI0} and the fact that $|A^\prime|\approx|A|$, we have 
$$
|A|^9\mu^{13}\theta^3\lessapprox  \delta^{-3}(\delta^3|\tubes|)^{12}.
$$
This establishes \eqref{trilineargoal} and completes the proof of Theorem \ref{GZTrilinearProp}.

\end{proof}
%
%
%
%


%
%
%

%
%
%

\section{The Planebrush argument}\label{planebrushSection}

In this section, we shall state the results from the paper of Katz and Zahl \cite{KZ} that we will be using. Proofs will be omitted for the sake of brevity and can be found in the aforementioned reference. First we need the following definition.

\begin{defn}[Planyness]
Let $(\tubes,Y)$ be a set of $\delta$-tubes and their associated shading. We say that $(\tubes,Y)$ is plany if for each $\delta$-cube $Q$, there is a plane $\Pi(Q)$ so that for all $T\in\tubes_Y(Q)$ we have $\angle(v(T),\Pi(Q))\lesssim\delta$.
\end{defn}

Next we need the following technical lemma which states that tubes are either plany or trilinear. This is Lemma 33 from Katz-Zahl \cite{KZ}.

\begin{lem}\label{planyOrTrilinearLem}
Let $(\tubes,Y)$ be a set of $\delta$-tubes and their associated shading. Suppose that $(\tubes,Y)$ is $(\epsilon_1,C_1)$-two-ends and $(\epsilon_2,C_2)$-robustly transverse.

Then there exists a number $\delta\leq\theta\leq 1$; a refinement $(\tubes^\prime,Y^\prime)$ of $(\tubes,Y)$; a set $\Omega\subset S^3$ of $\theta$-separated points; a set $(\tubes_{\theta},Y_\theta)$ of essentially distinct $\theta$-tubes and their associated shading; and numbers $B,B_1,B_2\geq 1$ with $B_1\geq B_2$ and $B_1/B_2 \leq B$ so that the following holds.
\begin{enumerate}[label=(\alph*)]
\item\label{deltaCubePlane} For each $Q\in\mathcal{Q}(Y^\prime)$, there is a plane $\Pi(Q)$ so that $\angle(v(T),\Pi(Q))\leq\theta$ for all $T\in\tubes_Y(Q)$.

\item\label{mostTriplesWedgeThetaItem} For each $Q\in\mathcal{Q}(Y^\prime)$,
\begin{equation}\label{mostTriplesWedgeTheta}
|\{T_1,T_2,T_3\in \tubes_{Y^\prime}(Q)\colon |v(T_1)\wedge v(T_2)\wedge v(T_3)|\gtrsim \theta\}|\gtrapprox|\tubes_{Y^\prime}(Q)|^3.
\end{equation}

\item\label{directionsThetaTubes} Every tube in $\tubes_\theta$ points in a direction from $\Omega$. For each $v\in\Omega$, there are $\sim|\tubes_\theta|/|\Omega|$ tubes from $\tubes_\theta$ that point in direction $v$.

\item\label{thinTubesInsideThetaTubes} Each tube $T_\theta \in\tubes_\theta$ contains $\sim |\tubes^\prime|/|\tubes_\theta|$ tubes from $\tubes^\prime$, and each tube from $\tubes^\prime$ is contained in exactly one tube from $\tubes_\theta$.

\item\label{fineShadingInsideCourseShading} If $T\in\tubes, T_\theta\in\tubes_\theta$, and $T\subset T_\theta$, then $Y^\prime(T)\subset Y_\theta(T_\theta)$.

\item\label{lowerBoundLambdaYTheta} $\lambda_{Y_\theta}\gtrapprox \lambda_{Y^\prime}$. 

\item\label{thetaTubesAreTwoEnds} $(\tubes_\theta,Y_\theta)$ is $(\eps_1,C_1^\prime)$-two-ends, with $C_1^\prime\lessapprox 1$.

\item\label{tubesHaveUniformMult} For each $\delta$-cube $Q\in\mathcal{Q}(Y)$, we have $|\tubes_{Y^\prime}(Q)|\sim \mu_{Y^\prime}$. For each $\theta$-cube $Q\in\mathcal{Q}(Y_\theta)$ we have $|\tubes_{\theta}(Q)|\sim \mu_{Y_\theta}$.

\item\label{innerTubesHaveUniformMult} There is a number $\mu_{\operatorname{fine}}$ so that for each $T_\theta\in\tubes_{\theta}$, we have $\mu_{Y_{T_\theta}}\sim\mu_{\operatorname{fine}}$. 

\item\label{thetaTubesInPlanes} For each $\theta$-cube $Q_\theta\in\mathcal{Q}(Y_\theta)$, there are planes $\Pi_1(Q_\theta),\ldots,\Pi_{B_1}(Q_\theta)$ and collections of tubes $\tubes_{1,\theta}(Q_\theta),\ldots,\tubes_{B_1,\theta} \subset\tubes_\theta(Q_\theta)$ so that for each index $i$, the tubes in $\tubes_{i,\theta}(Q_\theta)$ satisfy $\angle(T_\theta,\Pi_i(Q_\theta))\leq\theta$ and $|\tubes_{i,\theta}|=\mu_{Y_\theta}/B$. Each tube $T\in\tubes_{\theta}(Q)$ is contained in $B_2$ of sets $\{\tubes_{i,\theta}(Q_\theta)\}.$ 

\item\label{productMultiplicity} We have the multiplicity bound
\begin{equation*}
\mu_{Y^\prime}\lessapprox \mu_{Y_{\theta}}\mu_{\operatorname{fine}}/B.
\end{equation*}

\end{enumerate}
\end{lem}

The planebrush argument, roughly speaking, states that a plany Kakeya set in $\RR^4$ has dimension 10/3. This was proved by Katz and Zahl \cite{KZ}. We now state a technical version of this result which we will use. This is Proposition 34 from Katz-Zahl \cite{KZ}. Note that hypotheses have been chosen to match the conclusions of Lemma \ref{planyOrTrilinearLem}.

\begin{prop}\label{volumeBdPlainyTubes}
Let $0<\theta<1$ and let $0<\eps_1<1$. Let $\Omega\subset S^3$ be a set of $\theta$-separated directions. Let $(\tubes,Y)$ be a set of essentially distinct plany $\theta$-tubes and their associated shading. Suppose that 
\begin{itemize}
\item There are $\sim|\tubes|/|\Omega|$ tubes from $\tubes$ pointing in each direction $v\in\Omega$. 
\item $(\tubes,Y)$ is $(\eps_1,C_1)$-two-ends.
\item There are numbers $B,B_1,B_2$ with $B_1\geq B_2$ and $B_1/B_2 \leq B$ so that for each $Q\in\mathcal{Q}(Y)$, there are planes $\Pi_1(Q),\ldots,\Pi_{B_1}(Q)$ and collections of tubes $\tubes_{1}(Q),\ldots,\tubes_{B_1}(Q) \subset\tubes(Q)$ so that for each index $i$, the tubes in $\tubes_{i}(Q)$ satisfy $\angle(T,\Pi_i(Q))\leq\theta$ and $|\tubes_{i}|=\mu_{Y}/B$. Each tube $T\in\tubes(Q)$ is contained in $B_2$ sets from $\{\tubes_{i}(Q)\}.$ 
\end{itemize}
Then for each $\eps>0$, we have

\begin{equation}
\Big|\bigcup_{T\in\tubes}Y(T)\Big|\geq c C_1^{-C/\eps_1}\theta^{\eps}\lambda_Y^{4/3} \theta^{2/3} B^{-4/3}\big(\theta^3|\Omega|\big)^{1/3}\big(\theta^3|\tubes|\big)^{2/3}.
\end{equation}
Here $C$ is an absolute constant and $c>0$ depends on $\eps$ and $\eps_1$.
\end{prop}

\section{Proof of Theorem \ref{mainresult}} \label{finalproof}

In this section we will prove Theorem \ref{mainresult}. We will be borrowing the induction architecture from Katz-Zahl, adapting to the new $3/4$ exponent of $\theta$ (instead of 1 previously) in Corollary \ref{trilinearVolumeBound}. This improvement in the exponent yields the improved dimension bound. We begin with a lemma that lets us upgrade certain assertions of the form $\mathbf{TE}(d,a,b)$ to stronger assertions $\mathbf{TE}(d^\prime,a^\prime,b^\prime)$. Theorem \ref{mainresult} will eventually be proved by iterating this lemma.
 
\begin{lem}\label{boundWithCorrectTubesDependence}
Let $k=2.69$. Suppose that Assertion $\mathbf{TE}(4-\alpha, k, 1-\alpha/3)$ is true for some $0.84\leq\alpha\leq 0.97$. Define
\begin{equation}\label{defnOfAlphapPp}
\alpha^\prime =\frac{-76\alpha^2+24\alpha+189}{315-174\alpha},\quad \alpha^{\prime\prime}=\frac{3}{4}+\frac{63}{100\alpha-21}\bigg(\alpha-\frac{3}{4}\bigg).
\end{equation}
Then Assertion $\mathbf{TE}(4-\alpha^{\prime\prime},13/4, 1)$ is true. If $\alpha'\leq\alpha,$ then Assertion $\mathbf{TE}(4-\alpha^\prime,k, 1-\alpha^\prime/3)$ is true.
\end{lem}

\begin{proof}
To obtain the above assertions, it suffices (by Proposition \ref{robustTransReduction}) to prove that Assertion $\mathbf{RT}(4-\alpha^{\prime\prime},13/4,1)$ and Assertion $\mathbf{RT}(4-\alpha^\prime,k, 1-\alpha'/3)$ is true. For the latter, note that $\frac{(4-\alpha^\prime)-1}{3}= 1-\alpha'/3\leq 1$, so the hypotheses of Proposition \ref{robustTransReduction} are met. To this end, let $(\tubes,Y)$ be a set of $\delta$-tubes and their associated shading that is $(\eps_1,C_1)$-two-ends and $(\eps_2,100)$-robustly transverse, and let $\eps>0$. 

Apply Lemma \ref{planyOrTrilinearLem} to $(\tubes,Y)$ and let $\theta$, $\Omega$, $(\tubes^\prime,Y^\prime)$, $(\tubes_\theta,Y_\theta)$, and $\mu_{\operatorname{fine}}$ be the output from that Lemma. By Item \ref{tubesHaveUniformMult} from Lemma \ref{planyOrTrilinearLem}, we have that each cube $Q\in\mathcal{Q}(Y^\prime)$ satisfies $|\tubes_{Y^\prime}(Q)|\lesssim \mu_{Y^\prime}$. Thus
\begin{equation}\label{volumeBoundVsMultBound}
\begin{split}
\Big|\bigcup_{T\in\tubes}Y(T)| & \geq \Big|\bigcup_{T\in\tubes^\prime}Y^\prime(T)|\Big|\\
&\gtrsim \mu_{Y^\prime}^{-1}\lambda_{Y^\prime}(\delta^3|\tubes^\prime|)\\
&\gtrapprox \mu_{Y^\prime}^{-1}\lambda_{Y}(\delta^3|\tubes|),
\end{split}
\end{equation}
where on the final line we used the fact that $(\tubes^\prime,Y^\prime)$ is a refinement of $(\tubes,Y)$; this implies that $\lambda_{Y^\prime}(\delta^3|\tubes^\prime|)\gtrapprox \lambda_{Y}(\delta^3|\tubes|)$, and that 
\begin{equation}\label{lambdaYvsYPrime}
\lambda_Y\lessapprox \lambda_{Y^\prime}.
\end{equation} 
Thus it suffices to prove that
\begin{equation}\label{pointwiseMultBoundWrongTubeDependence}
\mu_{Y^\prime}\lessapprox \delta^{-\eps/2} C_1^{C/\eps_1}\min\Big(\lambda_{Y^\prime}^{-(k-1)}\delta^{-\alpha^\prime}(\delta^3|\tubes|)^{\alpha/3} ,\ \ \lambda_{Y^\prime}^{-9/4}\delta^{-\alpha^{\prime\prime}}\Big),
\end{equation}
where the constant $C$ may depend on $\alpha$, and the implicit constant may depend on $\eps$, $\eps_1,\eps_2$, and $\alpha$. Indeed, combining \eqref{volumeBoundVsMultBound}, \eqref{lambdaYvsYPrime}, and \eqref{pointwiseMultBoundWrongTubeDependence}, we obtain the volume bound
\begin{equation}
\Big|\bigcup_{T\in\tubes}Y(T)\Big|\gtrapprox \delta^{\eps/2}C_1^{-C/\eps_1}\max\Big(\lambda_Y^{k}\delta^{\alpha^\prime}(\delta^3|\tubes^\prime|)^{1-\alpha/3},\ \lambda_Y^{13/4}\delta^{\alpha^{\prime\prime}}(\delta^3|\tubes^\prime|)\Big).
\end{equation}
This volume bound is precisely what is needed to establish Assertion $\mathbf{RT}(4-\alpha^{\prime\prime},13/4,1)$, and it also establishes Assertion $\mathbf{RT}(4-\alpha^\prime,k, 1-\alpha^\prime/3)$ provided $\alpha'\leq\alpha$. 

The remainder of the proof will be devoted to establishing \eqref{pointwiseMultBoundWrongTubeDependence}. Define $A=|\tubes_\theta|/|\Omega|$ and define $B = \mu_{Y^\prime}/\mu_{\operatorname{fine}}$. From Proposition \ref{volumeBdPlainyTubes}, we have the bound
\begin{equation*}
\mu_{Y_\theta}\lessapprox C_1^{-C/\eps_1} \lambda_{Y_\theta}^{-1/3}\theta^{-2/3-\eps/2}A^{1/3}B^{4/3}.
\end{equation*}
From Corollary \ref{KTCor} we have the bound
\begin{equation*}
\begin{split}
\mu_{Y_\theta}&\lessapprox \lambda_{Y_\theta}^{-2}\theta^{-1-\eps/2}A^{3/4}(\theta^3 |\Omega|)^{1/3}\\
& = \lambda_{Y_\theta}^{-2}\theta^{-1-\eps/2}A^{3/4}(\theta^3 A^{-1}|\tubes_\theta|)^{1/3}
\end{split}
\end{equation*}
Combining these, we obtain the bound
\begin{equation}
\begin{split}
\mu_{Y_\theta}&\lessapprox \Big(C_1^{-C/\eps_1} \lambda_{Y_\theta}^{-1/3}\theta^{-2/3-\eps/2}A^{1/3}B^{4/3}\Big)^{3/4}
\Big( \lambda_{Y_\theta}^{-2}\theta^{-1-\eps/2}A^{3/4}(\theta^3 A^{-1}|\tubes_\theta|)^{1/3}\Big)^{1/4}\\
& \leq C_1^{-C/\eps_1} \lambda_{Y_\theta}^{-3/4}\theta^{-3/4-\eps/2}A^{7/16}B (\theta^3 A^{-1}|\tubes_\theta|)^{1/12}.
\end{split}
\end{equation}

From Items \ref{lowerBoundLambdaYTheta} and \ref{productMultiplicity} of Lemma \ref{planyOrTrilinearLem}, we have that 
\begin{equation}\label{pointwiseBdCorrectB}
\begin{split}
\mu_{Y^\prime}& \lessapprox \mu_{Y_\theta}\mu_{\operatorname{fine}}/B\\
&\lessapprox C_1^{-C/\eps_1} \lambda_{Y_\theta}^{-3/4}\theta^{-3/4-\eps/2}A^{7/16} (\theta^3 A^{-1}|\tubes_\theta|)^{1/12}\mu_{\operatorname{fine}}\\
&\lessapprox C_1^{-C/\eps_1} \lambda_{Y^\prime}^{-3/4}\theta^{-3/4-\eps/2}A^{7/16} (\theta^3 A^{-1}|\tubes_\theta|)^{1/12}\mu_{\operatorname{fine}}.
\end{split}
\end{equation}
Since the tubes whose shading contains a $\delta$-cube $Q\in\mathcal{Q}(Y^\prime)$ must lie in the $\theta$ neighborhood of a plane, we also have the bound

\begin{equation}\label{pointwiseBoundPlainyTrivial}
\mu_{Y^\prime}\leq \theta^{-1}\mu_{\operatorname{fine}}. 
\end{equation}

Combining \eqref{pointwiseBdCorrectB} and \eqref{pointwiseBoundPlainyTrivial}, we obtain the bound
\begin{equation}\label{pointwiseThetaBound}
\begin{split}
\mu_{Y^\prime}&\lessapprox  \Big(C_1^{-C/\eps_1} \lambda_{Y^\prime}^{-3/4}\theta^{-3/4-\eps/2}A^{7/16} (\theta^3 A^{-1}|\tubes_\theta|)^{1/12}\mu_{\operatorname{fine}}\Big)^{\frac{16\alpha}{21}}
\Big(\theta^{-1}\mu_{\operatorname{fine}}\Big)^{1-\frac{16\alpha}{21}}\\
&\lessapprox C_1^{-C/\eps_1}\lambda_{Y^\prime}^{\frac{-4\alpha}{7}}\theta^{\frac{4\alpha-21}{21}-\eps/2} A^{\frac{17\alpha}{63}} (\theta^3 |\tubes_\theta|)^{\frac{4\alpha}{63}}\mu_{\operatorname{fine}}.
\end{split}
\end{equation}

Since Assertion $\mathbf{TE}_\delta(4-\alpha,k, 1-\alpha/3)$ is true, we can apply it to the tubes contained in each of the $\theta$-tubes $T_\theta\in\tubes_\theta$. We conclude that
\begin{equation}\label{muFineEstimateInduction}
\begin{split}
\mu_{\operatorname{fine}}&\lesssim  C_1^{-C/\eps_1}\lambda_{Y^\prime}^{-(k-1)}(\delta/\theta)^{-\alpha-\eps/2}\Big((\delta/\theta)^3|\tubes|/|\tubes_\theta|\Big)^{\alpha/3}\\
&\lesssim C_1^{-C/\eps_1} \lambda_{Y^\prime}^{-(k-1)}(\delta/\theta)^{-\alpha-\eps/2}\Big((\delta/\theta)^3|\tubes|/|\tubes_\theta|\Big)^{4\alpha/63}A^{-17\alpha/63},
\end{split}
\end{equation}
where the implicit constant depends on $\eps,\eps_1,\eps_2,$ and $\alpha$. On the second line we used the fact that there are $\sim A$ tubes in $T_\theta$ pointing in each $\theta$-separated direction, and each of these $\theta$-tubes contain $\lesssim A^{-1}(\theta/\delta)^3$ tubes from $\tubes^\prime$. 

Combining \eqref{pointwiseThetaBound} and \eqref{muFineEstimateInduction}, we have 
\begin{equation}\label{boundInvolvingTheta}
\begin{split}
\mu_{Y^\prime}
&\lessapprox \Big(C_1^{-C/\eps_1}\lambda_{Y^\prime}^{\frac{-4\alpha}{7}}\theta^{\frac{4\alpha-21}{21}-\eps/2} A^{\frac{17\alpha}{63}} (\theta^3 |\tubes_\theta|)^{\frac{4\alpha}{63}}\Big) \\
&\quad\cdot\Big(C_1^{-C/\eps_1} \lambda_{Y^\prime}^{-(k-1)}(\delta/\theta)^{-\alpha-\eps/2}\Big((\delta/\theta)^3|\tubes|/|\tubes_\theta|\Big)^{4\alpha/63}A^{-17\alpha/63}  \Big)\\
&\leq C_1^{-C/\eps_1} \lambda_{Y^\prime}^{-\frac{4\alpha}{7} - (k-1)}\theta^{\frac{25\alpha-21}{21}}\delta^{-\alpha-\eps/2}(\delta^3|\tubes|)^{4\alpha/63}.
\end{split}
\end{equation}

By Corollary \ref{trilinearVolumeBound}, we also have the estimate 
\begin{equation}
\mu_{Y^\prime} \lessapprox  \lambda_{Y^\prime}^{-9/4}\theta^{-3/4}\delta^{-3/4-\eps/2}(\delta^3|\tubes|)^{3/4}.
\end{equation}
Interpolating these two estimates, we have
\begin{equation}\label{multiplicityBoundWrongPowerOfTubes}
\begin{split}
\mu_{Y^\prime}&\lessapprox \Big(C_1^{-C/\eps_1} \lambda_{Y^\prime}^{-\frac{4\alpha}{7} - (k-1)}\theta^{\frac{25\alpha-21}{21}}\delta^{-\alpha-\eps/2}(\delta^3|\tubes|)^{4\alpha/63}\Big)^{\frac{63}{100\alpha-21}} \\
&\quad\cdot\Big(\lambda_{Y^\prime}^{-9/4}\theta^{-3/4}\delta^{-3/4-\eps/2}(\delta^3|T|)^{3/4}\Big)^{1-\frac{63}{100\alpha-21}}\\
&\lessapprox \lambda_{Y^\prime}^{-\frac{9}{4}+\frac{63}{100\alpha-21}\big(\frac{9}{4}-\frac{4\alpha}{7}-(k-1)\big)} \delta^{-\frac{3}{4}+\frac{63}{100\alpha-21}\big(\frac{3}{4}-\alpha\big)-\frac{\epsilon}{2}}  (\delta^3|\tubes|)^{\frac{4\alpha}{100\alpha-21}+\frac{3}{4}\big(1-\frac{63}{100\alpha-21}\big)}.
\end{split}
\end{equation}

Observe that for $0.84\leq\alpha\leq 0.97$ and $k=2.69$ we can graphically check that
\begin{equation}\label{exponentvalid}
0\leq \frac{63}{100\alpha-21}\leq 1,
\end{equation}
and
\begin{equation}\label{ineqInvolvingPowerLambda}
    -\frac{9}{4}+\frac{63}{100\alpha-21}\bigg(\frac{9}{4}-\frac{4\alpha}{7}-(k-1)\bigg)\geq -\frac{9}{4},
\end{equation}
and
\begin{equation}
    \frac{4\alpha}{100\alpha-21}+\frac{3}{4}\bigg(1-\frac{63}{100\alpha-21}\bigg)>0.
\end{equation}
Thus we get 
\begin{equation}\label{fistMuYpEstimate}
\mu_{Y^\prime}\lessapprox C_1^{C/\eps_1}\delta^{-\eps/2}\lambda_{Y^\prime}^{-9/4}\delta^{-\alpha^{\prime\prime}}.
\end{equation}
This inequality is the second term in \eqref{pointwiseMultBoundWrongTubeDependence}, so we have proved half of the inequality. Our next task is to establish the remaining half of \eqref{pointwiseMultBoundWrongTubeDependence}.

From Theorem \ref{wolffThm}, we have the estimate
\begin{equation}
\mu_{Y^\prime} \lessapprox \lambda_{Y^\prime}^{-1} \delta^{-1} (\delta^3|\tubes|)^{1/2}.
\end{equation}
Interpolating these two estimates, we have

\begin{equation}
\begin{split}
\mu_{Y^\prime}&\lessapprox  \Big( \lambda_{Y^\prime}^{-\frac{9}{4}+\frac{63}{100\alpha-21}\big(\frac{9}{4}-\frac{4\alpha}{7}-(k-1)\big)} \delta^{-\frac{3}{4}+\frac{63}{100\alpha-21}\big(\frac{3}{4}-\alpha\big)-\frac{\epsilon}{2}}  (\delta^3|\tubes|)^{\frac{4\alpha}{100\alpha-21}+\frac{3}{4}\big(1-\frac{63}{100\alpha-21}\big)} \Big)^{ \frac{(4\alpha-6)(100\alpha-21)}{348\alpha-630}}\\
&\quad\cdot
\Big(\lambda_{Y^\prime}^{-1}\delta^{-1}(\delta^3|\tubes|)^{\frac{1}{2}} \Big)^{1-\frac{(4\alpha-6)(100\alpha-21)}{348\alpha-630}}\\
&\leq C_1^{-C/\eps_1}\lambda_{Y^\prime}^{\frac{(462-322\alpha-126k)(2\alpha-3)}{348\alpha-630}-1}  \delta^{\frac{76\alpha^2-24\alpha-189}{315-174\alpha}} (\delta^3|\tubes|)^{\alpha/3}.
\end{split}
\end{equation}

Note that if $0.84\leq\alpha\leq 0.97$ and $k=2.69$, then one can graphically check that
\begin{equation}
    0\leq\frac{(4\alpha-6)(100\alpha-21)}{348\alpha-630}\leq 1,
\end{equation}
and
\begin{equation}\label{ineqInvolvingPowerLambda2}
    \frac{(462-322\alpha-126k)(2\alpha-3)}{348\alpha-630}-1 >-(k-1).
\end{equation}
Thus 
\begin{equation}\label{secondMuYpEstimate}
\mu_{Y^\prime}\lessapprox C_1^{C/\eps_1} \delta^{-\eps/2}  \lambda_{Y^\prime}^{-(k-1)}\delta^{-\alpha^{\prime}}(\delta^3|\tubes|)^{\alpha/3}.
\end{equation}
The estimates \eqref{fistMuYpEstimate} and \eqref{secondMuYpEstimate} give us \eqref{pointwiseMultBoundWrongTubeDependence}.
\end{proof}

\begin{lem}\label{TwoEndsBoundCorrectTubeDependence}
Let $\alpha = \frac{1}{196}(291-\sqrt{10593})$, let $d_0 =\frac{13}{4}-\frac{12951-54\sqrt{10593}}{41972} $ and let $k=2.69$. Then for each $\eps>0$, 
\begin{equation}\label{Ass1IsTrue}
\operatorname{Assertion} \mathbf{TE}\big(4-\alpha-\eps,\ k,\ 1- \alpha/3\big)\quad\textrm{is true},
\end{equation} 
and
\begin{equation}\label{Ass2IsTrue}
\operatorname{Assertion}\mathbf{TE}\big(d_0-\eps,\ 13/4,\ 1\big)\quad\textrm{is true}.
\end{equation} 
\end{lem}
\begin{proof} We will begin with \eqref{Ass1IsTrue}. In Katz-Zahl, while proving (111), we cross below the threshold $\alpha=0.97$ since $\frac{257-\sqrt{17665}}{128}=0.964551...$, and we have $2.69>75/28$ also. Therefore Assertion $\mathbf{TE}_\delta(4-\alpha_1,k, 1-\alpha_1/3)$ holds for some $\alpha_1$ just below 0.97 which we can start with. For each $j\geq 1$ suppose that Assertion $\mathbf{TE}(4-\alpha_{j},k, 1-\alpha_j/3)$ is true, and define 
\begin{equation*}
\alpha_{j+1} = \frac{-76\alpha_j^2+24\alpha_j+189}{315-174\alpha_j}.
\end{equation*}
Since $\alpha_{j+1}\leq\alpha_j$, Lemma \ref{boundWithCorrectTubesDependence} implies that Assertion $\mathbf{TE}(4-\alpha_{j+1},k, 1-\alpha_{j+1}/3)$ is true.  Observe that $\alpha_j \searrow \alpha$. Thus for each $\eps>0$, there is an index $j$ so that $\alpha_j\leq \alpha+\eps$. We conclude that \eqref{Ass1IsTrue} holds.

For \eqref{Ass2IsTrue}, note that for each $\eps>0$, \eqref{Ass2IsTrue} follows from Lemma \ref{boundWithCorrectTubesDependence} and \eqref{Ass1IsTrue}. 
\end{proof}
We are now ready to prove Theorem \ref{mainresult}.
\begin{proof}[Proof of Theorem \ref{mainresult}]
This follows from \eqref{Ass2IsTrue}.
\end{proof}

\bibliographystyle{abbrv}

\end{document}